\documentclass[11pt]{article}

\newcounter{remarknum}
\makeatletter
\def\remark{\@ifnextchar[{\@with}{\@without}}
\def\@with[#1]{\refstepcounter{remarknum}\textit{Remark~\theremarknum.} \label{rem:#1}}
\def\@without{\refstepcounter{remarknum}\textit{Remark~\theremarknum.} }
\makeatother

\RequirePackage[a4paper, left=18mm, right=18mm, top=25mm, bottom=30mm]{geometry}
\RequirePackage[OT1]{fontenc}
\RequirePackage[british]{babel}
\RequirePackage{amssymb, amsmath, amsthm}
\allowdisplaybreaks[2]
\RequirePackage[sort&compress, numbers]{natbib} 

\DeclareRobustCommand{\NLprefix}[3]{#2}

\RequirePackage{hyperref}
\RequirePackage{enumerate}
\RequirePackage{color}
\RequirePackage[margin=1cm]{caption}
\RequirePackage{subcaption}
\RequirePackage{tikz}
\usepackage{tikz-qtree}
\usetikzlibrary{trees} 

\RequirePackage{calc}

\numberwithin{equation}{section}
\theoremstyle{plain}
\newtheorem{lemma}{Lemma}
\newtheorem{proposition}{Proposition}
\newtheorem{theorem}{Theorem}

\DeclareSymbolFont{bbold}{U}{bbold}{m}{n}
\DeclareSymbolFontAlphabet{\mathbbold}{bbold}

\newcommand{\numberthis}{\stepcounter{equation}\tag{\theequation}}
\newcommand{\alert}[1]{\textcolor{red}{#1}} 

\renewcommand{\cite}{\alert{cite[?]}}
\renewcommand\citet\Citet
\renewcommand\citep\Citep
\renewcommand\citealt\Citealt
\renewcommand\citealp\Citealp
\renewcommand\citeauthor\Citeauthor

\renewcommand{\a}{\alpha}
\newcommand{\A}{\mathcal{A}}
\renewcommand{\b}{\beta}
\renewcommand{\d}{\delta}
\newcommand{\dd}{\,\mathrm{d}}

\newcommand{\E}{\mathbb{E}}
\newcommand{\g}{\gamma}
\newcommand{\ind}{\mathbbold{1}}

\newcommand{\N}{\mathbb{N}}
\renewcommand{\O}{\mathcal{O}}
\renewcommand{\P}{\mathbb{P}}
\renewcommand{\r}{\rho}

\newcommand{\Z}{\mathbb{Z}}

\newcommand{\ntoinfty}{{n\rightarrow\infty}}

\renewcommand{\hat}{\widehat}
\renewcommand{\tilde}{\widetilde}

\def\gg{\mathrm{GI/GI/1}}
\def\gm{\mathrm{GI/M/1}}
\def\mg{\mathrm{M/GI/1}}
\def\mm{\mathrm{M/M/1}}
\def\fb{\mathrm{FB}}
\def\fifo{\mathrm{FIFO}}
\def\mlf{\mathrm{MLF}}
\def\opt{\mathrm{OPT}}

\def\rmlf{\mathrm{RMLF}}
\def\dirmlf{\mathrm{e\rmlf}}
\def\srpt{\mathrm{SRPT}}

\renewcommand{\O}{O}

\begin{document}
\title{Achievable Performance of Blind Policies in Heavy Traffic}
\author{
\begin{tabular}{ccccccccc}
Nikhil Bansal\textsuperscript{a,b} &&& Bart Kamphorst\textsuperscript{a} &&& Bert Zwart\textsuperscript{a,b} \\
\texttt{n.bansal@tue.nl} &&& \texttt{b.kamphorst@cwi.nl} &&& \texttt{b.zwart@cwi.nl}
\end{tabular} \\
\begin{tabular}{l}
\footnotesize{\textsuperscript{a} Centrum Wiskunde \& Informatica, P.O. Box 94079, 1090 GB Amsterdam, the Netherlands} \\
\footnotesize{\textsuperscript{b} Technische Universiteit Eindhoven, P.O. Box 513, 5600 MB Eindhoven, the Netherlands}
\end{tabular}
}
\date{\small{\today}}
\maketitle

\begin{abstract}
For a $\gg$ queue, we show that the average sojourn time under the (blind) Randomized Multilevel Feedback algorithm is no worse than that under the Shortest Remaining Processing Time algorithm times a logarithmic function of the system load. Moreover, it is verified that this bound is tight in heavy traffic, up to a constant multiplicative factor. We obtain this result by combining techniques from two disparate areas: competitive analysis and applied probability.
\end{abstract}

\textit{Keywords: $\gg$ queue, response time, heavy traffic, competitive ratio, blind policies, shortest remaining processing time, randomized multilevel feedback}

\section{Introduction}
One of the most relevant and widely studied measures of quality of service in a $\gg$ queue is the average sojourn time, also known as response time or flow time, defined as the average time spent by a job from its arrival in the system until its completion \citep{bansal2001analysis, bansal2005average, bansal2006handling, becchetti2004nonclairvoyant, boxma2007tails, lin2011heavy, nair2010tailrobust, nuyens2008preventing, wierman2005nearly, wierman2008scheduling, wierman2012tailoptimal}. We consider the most basic setting of a single machine with pre-emption, i.e.\ jobs can be interrupted arbitrarily and resumed later without any penalty. It is well-known that the Shortest Remaining Processing Time ($\srpt$) policy, that at any time works on the job with the least remaining processing time, is the optimal policy for every problem instance (or equivalently for every sample path) for minimizing the sojourn time \citep{schrage1968letter}. However, to run $\srpt$ one needs the exact knowledge of all job sizes. This information may not be available in many settings; specifically, jobs sizes may only be known approximately, or may not be known at all \citep{lu2004sizebased}. For this reason one may have to be content with more generally applicable policies.

In this paper we are interested in policies that do not require the knowledge of job sizes in their scheduling decisions. We refer to such policies as {\em blind} policies. More formally, in a blind policy the scheduler is only aware of the existence of a job and how much processing it has received thus far. The size of the job becomes known to the scheduler only when it terminates and leaves the system. Observe that the class of blind policies contains several well-studied policies, such as Processor Sharing (also known as Round Robin) \citep{kleinrock1976queueingcomputer}, Foreground-Background \citep{nuyens2008foreground}, and First In First Out \citep{asmussen2003applied}.

It is natural to ask how much this inability to use the knowledge of sizes can hurt performance. In particular, how much can the average sojourn time between $\srpt$ and an optimal blind policy differ for a given $\gg$ queue? As an illustration, let us consider the $\mm$ queue. In this setting, all blind policies are identical due to the memoryless nature of the job size distribution. More precisely, \citet{conway1967theory} state that any blind policy has an average sojourn time equal to $\E[B]/(1-\rho)$, where $\E[B]$ is the average job size and $\rho$ is the load of the system. On the other hand, if job sizes are known upon arrival, then \citet{bansal2005average} derives that the average sojourn time $\E[T_\srpt]$ under $\mm/\srpt$ is
\begin{equation}
 \E[T_\srpt] = (1+o(1)) \frac{1}{\log \left( \frac{e}{1-\rho}\right)} \frac{\E[B]}{1-\rho},
 \label{eq:mmsrpt}
\end{equation}
where $o(1)$ vanishes as $\rho$ approaches one. 
That is, the $\srpt$ policy performs better by a factor $\log (e/(1-\rho))$ in heavy traffic. So while $\srpt$ can be arbitrarily better than a blind policy for a $\mm$ queue as the load approaches one, this improvement factor grows quite mildly.

The performance of $\srpt$ as a function of load can be dramatically different for heavy-tailed distributions. \citet{bansal2006handling} and \citet{lin2011heavy} show that the growth factor of the average sojourn time in heavy traffic can be much smaller than $1/(1-\r)$ even in $\mg$ queues. For example, if the job sizes follow a Pareto($\b$) distribution with $\b\in(1,2)$, then the growth factor of the average sojourn time $\E[T_{\srpt}]$ is $\E[B] \log (1/(1-\rho))$, up to constant factors depending on $\b$. On the other hand, \citet{kleinrock1976queueingcomputer} states that Processor Sharing has an average sojourn time of $\E[B]/(1-\rho)$ in any $\mg$ queue. As this example illustrates, it is conceivable that for a general distribution, the gap between blind policies and $\srpt$ can be much larger then in the $\mm$ case.

Another subfield of computer science where the performance improvement of $\srpt$ over blind policies has been studied is {\em competitive analysis} \citep{borodin1998online, fiat1998online, pruhs2004online}, which generally regards worst case analyses of algorithms. The study of competitive analysis of blind scheduling policies was initiated by \citet{motwani1994nonclairvoyant}, who showed that no blind deterministic algorithm\footnote{ Note that $\srpt$ is deterministic, but not blind.} can have a better competitive ratio than $\Omega(m^{1/3})$ for the problem of minimizing the average sojourn time, where $m$ is the number of jobs in an instance. 
Motwani et al.\ also showed that no blind randomized algorithm can have a competitive ratio better than $\Omega(\log(m))$. In a significant breakthrough, \citet{kalyanasundaram2003minimizing} gave an elegant and non-trivial randomized algorithm that they called Randomized Multilevel Feedback ($\rmlf$) and proved that it has competitive ratio of at most $\O(\log (m) \log( \log(m)))$. Later, \citet{becchetti2004nonclairvoyant} showed that $\rmlf$ is in fact an $\O(\log(m))$-competitive randomized algorithm and hence the best possible (up to constant factors). Additional background on multilevel algorithms can be found in \citet{kleinrock1976queueingcomputer}, and an analysis of the average sojourn time under such algorithms is performed in \citet{aalto2006mean}.

The result in \citet{becchetti2004nonclairvoyant} is derived under the assumption that job sizes are bounded from below by a strictly positive constant, an assumption which is removed in this paper. This ``extended'' version of $\rmlf$ is denoted by $\dirmlf$.

The insights from applied probability and competitive analysis concerning the relation between blind policies and $\srpt$ can be combined when $m$ is taken as the number of jobs in a regeneration cycle, which has an expected value of the order $1/(1-\r)$. We make this precise in our main theorem and its proof. The main theorem shows that, for a $\gg$ queue, the gap between $\srpt$ and the best blind policy $\A$ for that system is at most $\log(1/(1-\rho))$, up to constant factors. More specifically, we show that this growth factor is a guaranteed upper bound on the gap between $\srpt$ and the $\dirmlf$ algorithm. That is, we show that
\begin{equation}
 \E[T_\A] \leq \E[T_\dirmlf] = \O \left( \log \left( \frac{1}{1-\r} \right) \E[T_{\srpt}] \right)
 \label{eq:introresult}
\end{equation}
as $\rho$ grows to one. 
We note that the implementation of the $\rmlf$ algorithm does not depend on the distributions of interarrival times and job sizes and is therefore applicable to every $\gg$ queue; this property may not hold for an optimal blind policy $\A$ in general.
%
%
%

The second main contribution of this paper is the proof of \eqref{eq:introresult} itself. It involves a novel combination of techniques from competitive analysis and applied probability. Using a renewal argument, we consider the average sojourn time $\E[T_{\rmlf}]$ of jobs in a general busy period, and subsequently distinguish two types of busy periods (small and large) by the number of jobs. For small busy periods, we apply a worst-case performance bound of $\rmlf$ from the study of competitive analysis. For large busy periods, we derive the heavy traffic behaviour of moments of two functionals: the busy period duration and the number of jobs in a busy period. In particular, we show that the $\kappa$-th moment of both of these functionals behaves like $\O((1-\rho)^{1-2\kappa})$ for $\kappa \geq 1$; a new result, which may facilitate future instances where competitive analysis and regenerative process theory can be combined to obtain information about algorithms under uncertainty. To prove these bounds, we rely on properties of ladder height distributions derived in \citet{asmussen2003applied} and \citet{lotov2002inequalities}.

This paper is organized as follows. A detailed model description and notation is introduced in Section~\ref{sec:preliminaries}. Section~\ref{sec:competitive} clarifies the concept of a competitive ratio and describes the $\rmlf$ algorithm. Additionally, Section~\ref{subsec:dirmlf} relaxes the constraints on $\rmlf$ while preserving the competitive ratio. The main result, Theorem~\ref{thm:main}, is presented in Section~\ref{sec:main}, whereas its proof is given in Section~\ref{sec:proofmain}. Propositions required for the main theorem are proven in Section~\ref{sec:moments}. Finally, Section~\ref{sec:conclusion} concludes the paper.

\section{Preliminaries}
\label{sec:preliminaries}
This section introduces a general framework for sequences of $\gg$ queues, so that we may analyse their limiting behaviour in further sections. In particular, the model allows for a heavy traffic analysis of the average sojourn time and various other functionals.

\subsection*{Sequence of queues}
Consider a sequence of $\gg$ queues, indexed by $n\geq 1$, where jobs arrive sequentially with independent and identically distributed (i.i.d.) sizes $B^{(n)}_i,i\in\{1,2,\ldots\},$ chosen from a distribution $F^{(n)}_B$. The jobs are then processed by a single server with unit speed. The time between two consecutive job arrivals is given by the i.i.d. interarrival times $A^{(n)}_i,i\in\{1,2,\ldots\},$ chosen from a distribution $F^{(n)}_A$. All jobs and interarrival times are assumed to be positive, i.e.\ the support of $F^{(n)}_A$ and $F^{(n)}_A$ is contained in $(0,\infty)$. For notational convenience, we define $A^{(n)}:=A^{(n)}_1$ and $B^{(n)}:=B^{(n)}_1$. 

Assume that there exist $\r_0 \in (0,1)$ and $A_{\min}>0$, both independent of the system index $n$, such that the inequalities $\r_0 A_{\min} \leq \r_0 \E[A^{(n)}] \leq \E[B^{(n)}] < \E[A^{(n)}]$ are respected. The load of the system is denoted by $\r^{(n)}:= \E[B^{(n)}]/\E[A^{(n)}]\in[\r_0,1)$, and is interpreted as the fraction of time that the server is busy. As is customary in the literature on heavy traffic analysis \citep{gromoll2004diffusion, puha2015diffusion}, we assume $\lim_{\ntoinfty} \r^{(n)} =1$. The mean amount of work that a server completes between two consecutive arrivals is represented by $\mu^{(n)}:= \E[A^{(n)}] - \E[B^{(n)}] = \E[A^{(n)}] (1-\r^{(n)})$.
 
Furthermore, we require that the interarrival times have finite variance for all $n$ and additionally that $\limsup_{\ntoinfty} \E[(A^{(n)})^2]<\infty$. Since a queue can only form when a job arrives to a non-empty system, we pose the final requirement that for some $\d>0$ and $\g>0$ independent of $n$, the system satisfies $\P(B^{(n)} - A^{(n)} \geq \d)\geq \g$.
\vspace{\baselineskip}

\textbf{Example 1.} In order to interpret some of our obtained results, one may compare them to a $\mg$ queue that is send into heavy traffic in a natural manner. Specifically, assume that both the $A_i$'s and $B_i$'s have unit mean and consider interarrival times $A^{(r)}_i=A_i/r,i\in\{1,2,\ldots\}, r\in(0,1)$. This model experiences a load of $\r = \E[B]/\E[A^{(r)}]=r$ and is exposed to heavy traffic as $r$ tends to one due to decreasing interarrival times. The model fits in the framework described above by letting $A^{(n)}_i = A_i/(1-1/n), B^{(n)}_i = B_i$ and $\r^{(n)} = 1-1/n$, and is referred to as the \textit{Example Model}. All further remarks on the Example Model are emphasised by superscripts $r$ for all related variables and functionals; e.g.\ the notation $A^{(r)}_i$ indicates the $i$-th interarrival time in this Example Model.

\subsection*{Queueing functionals}
The sojourn time of a job is the amount of time it spends in the system, namely its service completion time minus its arrival time. Given a scheduling policy $\pi$, we denote the average sojourn time of a generic job by $\E[T^{(n)}_{\pi}]$. The steady-state cumulative amount of work in the system is represented by $V^{(n)}$, whose distribution has an atom at zero that corresponds to the times when the server is idle. The steady-state duration of such an idle period is denoted by $I^{(n)}$. Idle periods are ended by the arrival of a new job, which initiates a busy period. A busy period finishes at the earliest subsequent time for which the system is empty again. The steady-state duration of a busy period is represented by $P^{(n)}$, whereas the total number of arrivals during a busy period is denoted by $N^{(n)}$. Finally, the steady-state cumulative amount of work in the system \textit{at an arrival instance} is represented by $W^{(n)}$.

\subsection*{Scheduling policies}
A scheduling policy $\pi$ is an algorithm or a rule which specifies which job receives service at any time in the system. For the $\gg$ queue under consideration, such a policy prescribes the behaviour of a single server under the relaxation that jobs can be \textit{pre-empted}; that is, jobs can be interrupted at any point during their execution and can be resumed later from this point without any penalty. Of the large class of scheduling policies that apply to this system, we consider only those policies $\pi$ that satisfy the following two criteria (quoted from \citet{wierman2012tailoptimal}, after \citet{stolyar2001largest}):
\begin{enumerate}
 \item $\pi$ is \textit{non-anticipative}: a scheduling decision at time $t$ does not depend on information about customers that arrive beyond time $t$.
 \item $\pi$ is \textit{non-learning}: the scheduling decisions cannot depend on information about previous busy periods. That is, a scheduling decision on a sample path cannot change when the history before the current busy period is changed.
\end{enumerate}
Of special interest are those scheduling policies $\pi$ that additionally obey the following characteristic:
\begin{enumerate}
 \setcounter{enumi}{2}
 \item $\pi$ is \textit{blind}: the scheduling decisions do not depend on the sizes of the jobs. That is, the scheduling decisions on a sample path up to time $t$ cannot change when the sizes of jobs that have not finished at that time are altered (in such a way that the jobs remain unfinished).
\end{enumerate}
Policies that satisfy all above criteria are very common: First In First Out ($\fifo$), Processor Sharing and Foreground-Background ($\fb$) are all blind policies within the specified subclass of scheduling policies. On the other hand, policies like Shortest Job First or Shortest Remaining Processing Time ($\srpt$) are {\em non}-blind elements of the specified subclass as they require knowledge of the job sizes when making a scheduling decision. 

We let $\A^{(n)}$ denote a blind policy that minimizes the average sojourn time over the space of all blind policies for the $n$-th $\gg$ queue. In general, $\A^{(n)}$ could depend on the distributions $F^{(n)}_A$ and $F^{(n)}_B$ that specify the $\gg$ queue. 
We denote $\rmlf$ and $\dirmlf$ to be the randomized blind algorithms that are formalized in Sections~\ref{subsec:rmlf} and \ref{subsec:dirmlf}, respectively. The implementation of the $\rmlf$ and $\dirmlf$ algorithms does not depend on $F^{(n)}_A$ and $F^{(n)}_B$ and is therefore independent of the system index $n$.

Finally, we call a scheduling policy $\pi$ \textit{work-conserving} if it always has the server working at unit speed whenever work is present in the system. One can easily verify that all above policies, including $\A^{(n)}$, are work-conserving.

\subsection*{Asymptotic relations}
We use the standard notation that for two functions $f(n)$ and $g(n)$, $f(n) = \O(g(n))$ and $f(n) = o(g(n))$ if $\limsup_{\ntoinfty} f(n)/g(n) <\infty$ and $\limsup_{n\rightarrow \infty} f(n)/g(n) = 0$, respectively. Similarly, $f(n) = \Omega(g(n))$ means $\liminf_{\ntoinfty} f(n)/g(n) > 0$ and $f(n) = \Theta(g(n))$ is equivalent to $0 < \liminf_{\ntoinfty} f(n)/g(n) \leq \limsup_{\ntoinfty} f(n)/g(n) < \infty$.

Final notational conventions in this paper are the floor-function $\lfloor x \rfloor := \sup\{m\in\N: m\leq x\}$ and the indicator function $\ind(\textit{[logical expression]})$ that takes value $1$ if the logical expression is true, and value $0$ otherwise. 

\section{Competitive analysis of scheduling policies}
\label{sec:competitive}
In this section, we describe some relevant definitions and results from the area of competitive analysis, which deals with the worst case analysis of algorithms. We restrict our presentation here only to the competitive analysis of scheduling algorithms with respect to average sojourn time. Subsequently, we introduce the original $\rmlf$ algorithm and its extension $\dirmlf$.

A scheduling problem instance $\mathcal{I}$ consists of a collection of jobs specified by their sizes and their arrival times. We say that an instance has size $m$, i.e.\ $|\mathcal{I}|=m$, if it consists of $m$ jobs. For an instance $\mathcal{I}$, we denote the optimal average sojourn time possible for this instance by $\E[T_{\opt}(\mathcal{I})]$, which for our purposes is same as $\E[T_{\srpt}(\mathcal{I})]$. 

For a blind deterministic algorithm $\pi$, we let $\E[T_\pi(\mathcal{I})]$ denote the average sojourn time when the instance $\mathcal{I}$ is executed according to the algorithm $\pi$. We say that the algorithm $\pi$ has competitive ratio $c(m)$ if
\[
 \sup_{\mathcal{I}: |\mathcal{I}| \leq m} \frac{\E[T_\pi(\mathcal{I})]}{\E[T_{\opt}(\mathcal{I})]} \leq c(m).
\]

Thus, the competitive ratio of an algorithm (possibly a function of $m$), is the worst case ratio over all input instances of length at most $m$ of the sojourn time achieved by $\pi$ and the optimal sojourn time on that instance. 
Observe that the definition of the competitive ratio is rather strict, in that even if an algorithm is close to optimal on all but one input instance, its competitive ratio will be lower bounded by its performance on the bad input instance. 

For this purpose it is useful to consider randomized algorithms. A blind randomized algorithm $\tilde{\pi}$ can toss coins internally and base its decisions on this outcome of these internal random variables. Such an algorithm can thus be viewed as a probability distribution over blind deterministic algorithms $\pi_i$ \citep{borodin1998online}.
It then follows that the average sojourn time of instance $\mathcal{I}$ under a randomized algorithm $\tilde{\pi}$ equals $\E[T_{\tilde{\pi}}(\mathcal{I})]=\E_i[\E[T_{\pi_i}(\mathcal{I})]]$, where the outer expectation is over the internal random choices of the algorithm. We say that  
$\tilde{\pi}$ has competitive ratio $c(m)$ if
\begin{equation}
 \sup_{\mathcal{I}: |\mathcal{I}| \leq m} \frac{\E[T_{\tilde{\pi}}(\mathcal{I})]}{\E[T_{\opt}(\mathcal{I})]} 
 = \sup_{\mathcal{I}: |\mathcal{I}| \leq m} \frac{\E_i[\E[T_{\pi_i}(\mathcal{I})]]}{\E[T_{\opt}(\mathcal{I})]} 
 \leq c(m).
\end{equation}
Observe that the expectation is only over the random choices made by the algorithm, and the competitive ratio is still determined by the worst possible instance.
However, the competitive ratio of a blind randomized algorithm can be substantially lower, e.g.\ in situations where no single blind deterministic algorithm is good for all instances, but a suitable combination of algorithms is close to optimal for all instances.

\subsection{Randomized Multilevel Feedback algorithm}
\label{subsec:rmlf}
This section introduces \citet{kalyanasundaram2003minimizing}'s Randomized Multilevel Feedback ($\rmlf$) algorithm. As the name suggests, it is a randomized version of the Multilevel Feedback ($\mlf$) algorithm proposed by \citet{corbato1962experimental}. Both algorithms are blind and can therefore only learn the size of a job upon completion.

The general idea of both $\mlf$ and $\rmlf$ is to prioritize potential short jobs (e.g.\ jobs that have not received much service) and reduce the priority of a job as it receives more service. This prioritisation is embodied by assigning every job $J_j$ to a virtual high priority queue $Q_i$, and move it to a lower priority queue $Q_{i+1}$ once it has received $U_{i,j}$ units of service. The performance of the algorithm may suffer from a poor choice of the so-called targets $U_{i,j}$; in particular, if the job sizes are slightly above the targets, then jobs are moved to lower priority queues just prior to completion. The improvement of $\rmlf$ over $\mlf$ is due to randomization of the targets, thereby reducing the possibility of such events over general instances.

We now provide a mathematical representation of the $\rmlf$ algorithm. Assume first that there is a universal lower bound of on the job sizes in every instance $\mathcal{I}$, say with value $2$. For every instance of size $m$, the $j$-th job $J_j$ is released at time $r_j$ and has size $B_j$. The process $w_j(t)$ denotes the amount of time that $\rmlf$ has run $J_j$ before time $t$. For some symbolic constant $\theta$, fixed at $\theta:=12$, we define the independent exponentially distributed variables $\beta_j$ with $\P(\beta_j\leq x) = 1-\exp[-\theta x \ln j]$. Finally, the targets are defined as $U_{i,j} = 2^i \max\{1,2-\beta_j\}$ for all $i\in\{1,2,\ldots\}, j\in\{1,\ldots,m\}$. $\rmlf$ is then formalized in Figure~\ref{fig:rmlf}, identical to \citet{kalyanasundaram2003minimizing}.


\begin{figure}
 \begin{center}
 \fbox{
\begin{minipage}{.85\textwidth}
\small{\textbf{Algorithm $\rmlf$:} At all times the collection of released, but uncompleted, jobs are partitioned into queues, $Q_0,Q_1,\ldots$ We say that $Q_i$ is lower than $Q_j$ for $i<j$. [For] each job $J_j\in Q_i, U_{i,j}\in [2^i,2^{i+1}]$ when it entered $Q_i$. $\rmlf$ maintains the invariant that it is always running the job at the front of the lowest non-empty queue.
\\
When a job $J_h$ is released at time $r_h$, $\rmlf$ takes the following actions:
\begin{itemize}
\item Job $J_h$ is enqueued on $Q_0$.
\item The target $U_{0,h}$ is set to $\max\{1,2-\beta_h\}$.
\item If, just prior to $r_h$, it was the case that $Q_0$ was empty, and that $\rmlf$ was running a job $J_j$, $\rmlf$ then takes the following actions:
 \begin{itemize}
 \item Job $J_j$ is pre-empted. Note that $J_j$ remains at the front of its queue.
 \item $\rmlf$ begins running $J_h$.
 \end{itemize}
\end{itemize}
If at some time $t$, a job $J_j\in Q_{i-1}$ is being run when $w_j(t)$ becomes equal to $U_{i-1,j}$, then $\rmlf$ takes the following actions:
\begin{itemize}
\item Job $J_j$ is dequeued from $Q_{i-1}$.
\item Job $J_j$ is enqueued on $Q_i$.
\item The target $U_{i,j}$ is set to $2U_{i-1,j} = 2^i\max\{1,2-\beta_j\}$.
\end{itemize}
Whenever a job is completed, it is removed from its queue.}
\end{minipage}
 }
 \end{center}
\caption{Formal statement of $\rmlf$ algorithm.}\label{fig:rmlf}
\end{figure}

\citet{kalyanasundaram2003minimizing} proved that the $\rmlf$ algorithm is $\O(\log(m)\log(\log(m)))$- competitive. This result was later strengthened by \citet{becchetti2004nonclairvoyant} to a competitive ratio of $\O(\log(m))$:
\begin{theorem}
\label{thm:rmlf}
 The $\rmlf$ algorithm is $\log(m)$-competitive. That is,
 \begin{equation}
   \E[T_\rmlf(\mathcal{I})] \leq C_1 \log(m) \E[T_\srpt(\mathcal{I})]
 \end{equation}
 for all instances $\mathcal{I}$ of size at most $m$ and a universal constant $C_1$.
\end{theorem}
The competitive ratio lower bound of $\Omega(\log(m))$ as shown by \citet{motwani1994nonclairvoyant} implies that, up to multiplicative factors, this is the best bound possible for randomized algorithms in the current model. Note that this competitive ratio is significantly lower than the best possible ratio for blind deterministic algorithms: $\Omega(m^{1/3})$.

In the next section we propose a variant on $\rmlf$ that makes the assumption of a universal lower bound on job sizes obsolete.

\subsection{Extending the RMLF algorithm}
\label{subsec:dirmlf}
In a general $\gg$ queue there may not be a strictly positive lower bound on the job sizes. The $\rmlf$ algorithm is not directly applicable in that case. This problem is solved in an extension of the $\rmlf$ algorithm, which we will refer to as the $\dirmlf$ algorithm. The $\dirmlf$ algorithm defines queues $\tilde{Q}_1,\tilde{Q}_2,\ldots$ that are identical to the queues $Q_1,Q_2,\ldots$ of the $\rmlf$ algorithm, but splits the first queue $Q_0$ into many queues $\tilde{Q}_0,\tilde{Q}_{-1},\ldots$ Additionally, it considers a ``new job'' queue $\tilde{Q}^*$. The concept of the $\dirmlf$ algorithm is described below; the formal statement is presented in the appendix. 

Let a problem instance $\tilde{\mathcal{I}}$ for $\dirmlf$ be given. A target $\tilde{U}_{*,j}=2^{z^*_j} \max\{1,2-\tilde{\beta}_j\}$ is assigned to every job $\tilde{J}_j$ upon arrival, where $\tilde{\beta}_j$ is an exponentially distributed random variable and $z^*_j \in\Z$ depends on the current state of the system. When the target has been assigned to the new job, it receives service in $\tilde{Q}^*$ until either the job is completed, the obtained service equals the target, or a new job arrives. Once either of the latter two events happens, the job in $\tilde{Q}^*$ is assigned to a queue $\tilde{Q}_z, z\in \Z$.

If there are no jobs in queue $\tilde{Q}^*$, the $\dirmlf$ algorithm serves the queues $\tilde{Q}_z$ in a similar fashion as the $\rmlf$ algorithm. Moreover, at any time the problem instance $\tilde{\mathcal{I}}$ can be converted to a problem instance $\mathcal{I}$ for $\rmlf$ by a scaling argument, and under this scaling the sojourn times of all jobs are identical for both algorithms. From this perspective, it is only natural that $\dirmlf$ inherits the competitive ratio of $\rmlf$:

\begin{theorem}
\label{thm:dirmlf}
 The $\dirmlf$ algorithm is $\log(m)$-competitive. That is,
 \begin{equation}
   \E[T_\dirmlf(\mathcal{I})] \leq C_1 \log(m) \E[T_\srpt(\mathcal{I})]
   \label{eq:dirmlf}
 \end{equation}
 for all instances $\mathcal{I}$ of size at most $m$ for a universal constant $C_1$. This constant is identical to the constant $C_1$ in Theorem~\ref{thm:rmlf}.
\end{theorem}
The proof of Theorem~\ref{thm:dirmlf} is given in the appendix. 




\section{Main result}
\label{sec:main}
We are now ready to present the main result, Theorem~\ref{thm:main}. The main result states that the average sojourn time under $\srpt$ is at most a factor $\log (1/(1-\r^{(n)}))$ better than that under $\dirmlf$ in heavy traffic:
\begin{theorem}
\label{thm:main}
For a $\gg$ queue, the $\dirmlf$ algorithm satisfies the relation
\begin{equation}
 \E[T^{(n)}_\dirmlf] = \O \left( \log \left( \frac{1}{1-\r^{(n)}} \right) \E[T^{(n)}_{\srpt}] \right)
 \label{eq:main}
\end{equation}
as $\ntoinfty$, provided that $\sup_{n\in\{1,2,\ldots\}} \E[(B^{(n)})^\a]<\infty$ for some $\a>2$. 
\end{theorem}
The proof of the theorem is postponed until the next section. It relies on techniques from both competitive analysis and applied probability.

As a consequence of Theorem~\ref{thm:main}, the blind policy $\A^{(n)}$ that minimizes the average sojourn time for the $n$-th system also satisfies the above performance bound. We emphasise the fact that the implementation of $\dirmlf$ does not depend on the interarrival and job size distributions, whereas this may not be true for the optimal blind policy $\A^{(n)}$. This property may pose a considerable advantage over a system-dependent optimal blind policy with similar mean performance, for example when the input distributions are only approximately known. Also, we note that Theorem~\ref{thm:main} remains true if $\dirmlf$ is replaced by $\rmlf$, provided that the support of the job size distribution $F^{(n)}_B$ is uniformly bounded away from zero (i.e.\ $B_i^{(n)}\geq B_{\min}$ for some $B_{\min}>0$ independent of $i$ and $n$).

We conclude this section with some remarks:

\remark{
Recall that the average sojourn time under any blind policy in an $\mm$ queue is $\E[B^{(n)}]/(1-\r^{(n)})$, whereas the average sojourn time under $\srpt$ \citep{bansal2001analysis} is
\begin{equation}
 \E[T_\srpt^{(n)}] = (1+o(1)) \frac{1}{\log\left( \frac{e}{1-\rho^{(n)}}\right)} \frac{\E[B^{(n)}]}{1-\rho^{(n)}}.
\end{equation}
In this case, our result is tight up to a multiplicative factor.
}

\remark{
There may be sequences of $\gg$ queues for which $\E[T_\dirmlf^{(n)}]$ has a worse heavy traffic scaling than $\E[T_{\mathcal{A}^{(n)}}^{(n)}]$. For example, it is known that the $\fb$ policy minimizes the average sojourn time over all blind policies in a $\mg$ queue if $F^{(n)}_B$ has a decreasing failure rate \citep{righter1990extremal}. Moreover, if $F^{(n)}_B(x)=1-x^{-\b}, x\geq 1, \b\in(1,\infty)\slash \{2\},$ then $\E[T_\fb^{(n)}] = \Theta(\E[T_\srpt^{(n)}])$ displays the best possible scaling in heavy traffic \citep{nuyens2008foreground, lin2011heavy}. The heavy traffic behaviour of $\E[T_\dirmlf^{(n)}]$ is unknown for any $\gg$ queue and could scale worse than $\E[T_\fb^{(n)}]$ (although no worse than $\log(1/(1-\r))\E[T_\fb^{(n)}]$ by Theorem~\ref{thm:main}).

On the other hand, the optimal blind policy $\mathcal{A}^{(n)}$ may not be robust under different input distributions $F^{(n)}_A$ and $F^{(n)}_B$. Continuing the $\fb$ example, we see that it is optimal if $F^{(n)}_B$ is the Pareto distribution, yet $\E[T_\fb^{(n)}]=\Theta((1-\r)^{-2})=\Theta((1-\r)^{-1})\E[T_\srpt^{(n)}]$ if $F^{(n)}_B=\ind(x\leq 1)$ is deterministic \citep{nuyens2008foreground, lin2011heavy}.
}

\section{Proof of the main theorem}
\label{sec:proofmain}
The current section presents the proof of Theorem~\ref{thm:main}.

\subsection{Proof strategy.}
The competitive ratio of the $\dirmlf$ algorithm provides an upper bound on the suboptimality of $\dirmlf$. 
Specifically, it guarantees an upper bound of $\O(\log(m))$ on the ratio of the average sojourn time under $\dirmlf$ over the average sojourn time under $\srpt$, for instances of length at most $m$. Unfortunately, a general $\gg$ queue corresponds to an infinite-length problem instance and hence the competitive ratio result can not be applied directly.

The key idea of the proof is that a $\gg$ queue is a regenerative process, and as such one would like to analyse individual busy periods rather than the infinite problem instance. This approach is justified by the fact that for a single server, any two work-conserving scheduling policies $\pi_1$ and $\pi_2$ generate the same busy periods, i.e.\ $V_{\pi_1}^{(n)}(t) \equiv V_{\pi_2}^{(n)}(t)$. This means that the server is simultaneously active under both policies, and hence in particular that every busy period under $\pi_1$ can be compared to the same busy period under $\pi_2$.

Still, regarding every busy period as an individual problem instance does not bound the problem instance length. One way to circumvent the unbounded problem instances is by discriminating between ``small'' busy periods with at most $N^{(n)}_0$ jobs, and ``large'' busy periods. Busy periods with at most $N^{(n)}_0$ jobs can be analysed with the competitive ratio, yielding a bound of $\O(\log (N^{(n)}_0))$. This leaves us with the analysis of large busy periods.

Since the $\gg$ queue induces a distribution over problem instances, the probability of experiencing busy periods with at least $N^{(n)}_0$ jobs can be made arbitrarily small by choosing the threshold $N^{(n)}_0$ properly. The combined sojourn time of all the jobs in such a large busy period is dominated by the product of the number of jobs $N^{(n)}$ in the busy period and the duration $P^{(n)}$ of the busy period. Therefore, the contribution of large busy periods to the overall average sojourn time is at most $\E[N^{(n)}P^{(n)}\ind(N^{(n)}>N^{(n)}_0)]/\E[N^{(n)}]$. We will show that, for an appropriate choice of $N^{(n)}_0$, the contribution of the large busy periods is $o(\log( N^{(n)}_0))$.

The second part of this section formalizes the above strategy. In the analysis of the expectation $\E[N^{(n)}P^{(n)}\ind(N^{(n)}>N^{(n)}_0)]$ we greatly rely on H{\"o}lder's inequality for decoupling the given expectation into individual moments of $P^{(n)}$ and $N^{(n)}$. The behaviour of these moments is then the subject of Propositions~\ref{prop:Pkappa} and \ref{prop:Nkappa}, both of which are proven in Section~\ref{sec:moments}.

\subsection{Small and large busy periods}
We begin by specifying the threshold that distinguishes small and large busy periods based on the number of jobs. Fix $s\in\left(\frac{\a}{\a-1},2\right)$ and $\zeta > \frac{4+2s}{2-s}$. The threshold $N^{(n)}_0$ is now defined as $N^{(n)}_0:=(1-\r^{(n)})^{-\zeta}$.

Let $T^{(n)}_{\dirmlf,i}, T^{(n)}_{\srpt,i}, i\in\{1,\ldots,N^{(n)}\},$ be the sojourn time of job $i$ under algorithm $\dirmlf$ and $\srpt$, respectively. Using the fact that a $\gg$ queue is a regenerative process, we only need to consider a general busy period when analysing the average sojourn time \citep[Thm.~VI.1.2, Prop.~X.1.3]{asmussen2003applied}:
\begin{equation}
 \E[T^{(n)}_{\dirmlf}] = \frac{1}{\E[N^{(n)}]} \E\left[ \sum_{i=1}^{N^{(n)}} T^{(n)}_{\dirmlf, i} \right].
\end{equation}
Discriminating between small and large busy periods then yields
\begin{equation}
 \E[T^{(n)}_{\dirmlf}] = \frac{1}{\E[N^{(n)}]} \E\left[ \sum_{i=1}^{N^{(n)}} T^{(n)}_{\dirmlf, i} \ind(N^{(n)}\leq N^{(n)}_0) \right]
  + \frac{1}{\E[N^{(n)}]} \E\left[ \sum_{i=1}^{N^{(n)}} T^{(n)}_{\dirmlf, i} \ind(N^{(n)}> N^{(n)}_0) \right].
 \label{eq:busyperiods}
\end{equation}
As described in the strategy, we will bound the first term by means of the competitive ratio of $\dirmlf$ and show that the second term vanishes asymptotically as $n\rightarrow \infty$. These analyses are the subjects of the following two subsections.

\subsection{Small busy periods: competitive ratio}
The first term in \eqref{eq:busyperiods} considers busy periods with at most $N_0^{(n)}$ jobs. Theorem~\ref{thm:dirmlf} ensures that, for any problem instance $\mathcal{I}$ with $N^{(n)}\leq N_0^{(n)}$ jobs, the average sojourn time $\E[T_{\dirmlf}(\mathcal{I})]$ is bounded by $C_1 \log(N_0^{(n)}) \E[T_{\srpt}(\mathcal{I})]$. In particular
\begin{align*}
 \frac{1}{\E[N^{(n)}]} \E\left [\sum_{i=1}^{N^{(n)}} T^{(n)}_{\dirmlf,i} \ind(N^{(n)} \leq N^{(n)}_0) \right]
  &\leq \frac{C_1}{\E[N^{(n)}]} \log (N^{(n)}_0) \E\left[ \sum_{i=1}^{N^{(n)}} T^{(n)}_{\srpt,i} \ind(N^{(n)} \leq N^{(n)}_0) \right] \\
  &\leq \frac{C_1}{\E[N^{(n)}]} \log (N^{(n)}_0) \E\left[ \sum_{i=1}^{N^{(n)}} T^{(n)}_{\srpt,i} \right] \\
  &= C_1 \log (N^{(n)}_0) \E[T^{(n)}_{\srpt}].
\end{align*}
The proof is complete once we show that the second term in~\eqref{eq:busyperiods} is dominated by $\log(N_0^{(n)}) \E[T^{(n)}_{\srpt}]$ as $\ntoinfty$.

\subsection{Large busy periods: H{\"o}lders inequality}
For any work-conserving scheduling policy, the sojourn time of an individual job is bounded by the duration $P^{(n)}$ of the busy period. Therefore, the second term in \eqref{eq:busyperiods} is bounded by
\begin{equation}
 \frac{1}{\E[N^{(n)}]} \E\left[\sum_{i=1}^{N^{(n)}} T^{(n)}_{\dirmlf, i} \ind(N^{(n)}> N^{(n)}_0) \right] \leq \frac{1}{\E[N^{(n)}]} \E\left[ N^{(n)} P^{(n)} \ind(N^{(n)}> N^{(n)}_0) \right].
\end{equation}
The functionals $N^{(n)}$ and $P^{(n)}$ are dependent, which makes an exact analysis of the expectation troublesome. This complication is avoided by applying H{\"o}lder's inequality, which allows us to approximate the dependent expectation by the product of two expectations. In particular, for $\tilde{s}=\frac{s}{s-1}\in (2,\a)$ we have $\frac{1}{\tilde{s}}+\frac{1}{s}=1$ and hence
\begin{equation}
 \E\left[\sum_{i=1}^{N^{(n)}} T^{(n)}_{\dirmlf, i} \ind(N^{(n)}> N^{(n)}_0) \right] \leq \E[(P^{(n)})^{\frac{s}{s-1}}]^{\frac{s-1}{s}} \E[(N^{(n)})^s \ind(N^{(n)}>N^{(n)}_0)]^{\frac{1}{s}}.
\end{equation}
Applying H{\"o}lder's inequality once more with parameters $\frac{2}{s}$ and $\frac{2}{2-s}$, we get
\begin{equation}
 \E\left[\sum_{i=1}^{N^{(n)}} T^{(n)}_{\dirmlf, i} \ind(N^{(n)}> N^{(n)}_0) \right] \leq \E[(P^{(n)})^\frac{s}{s-1}]^\frac{s-1}{s} \E[(N^{(n)})^2]^\frac{1}{2} \P(N^{(n)}>N^{(n)}_0)^\frac{2-s}{2s}.
\end{equation}
Finally, the tail probability of $N^{(n)}$ is bounded by Markov's inequality. We therefore obtain the following upper bound for the second term in \eqref{eq:busyperiods}:
\begin{equation}
 \frac{1}{\E[N^{(n)}]} \E\left[\sum_{i=1}^{N^{(n)}} T^{(n)}_{\dirmlf, i} \ind(N^{(n)}> N^{(n)}_0) \right] \leq \E[(P^{(n)})^\frac{s}{s-1}]^\frac{s-1}{s} \E[(N^{(n)})^2]^\frac{1}{2} \frac{\E[N^{(n)}]^{\frac{2-s}{2s}-1}}{(N^{(n)}_0)^\frac{2-s}{2s}}.
\end{equation}

The analysis of the average sojourn time for large busy periods is now reduced to the analysis of moments of $N^{(n)}$ and $P^{(n)}$. The following two propositions quantify the behaviour of these moments.
\begin{proposition}
\label{prop:Pkappa}
Assume $\sup_{n\in\{1,2,\ldots\}} \E[(B^{(n)})^\alpha]<\infty$ for some $\a\geq 2$. Then
\begin{equation}
 \E[(P^{(n)})^\kappa] = \O\left((1-\r^{(n)})^{1-2\kappa} \right).
 \label{eq:Pkappa}
\end{equation}
for all $\kappa \in [1,\alpha]$. Moreover, $\E[P^{(n)}] = \Theta \left((1-\r^{(n)})^{-1}\right)$.
\end{proposition}
\begin{proposition}
\label{prop:Nkappa}
Assume $\sup_{n\in\{1,2,\ldots\}} \E[(B^{(n)})^\alpha]<\infty$ for some $\a\geq 2$. Then
\begin{equation}
 \E[(N^{(n)})^\kappa] = \O\left((1-\r^{(n)})^{1-2\kappa}\right).
 \label{eq:Nkappa}
\end{equation}
for all $\kappa \in [1,\alpha]$. Moreover, $\E[N^{(n)}] = \Theta \left((1-\r^{(n)})^{-1}\right)$.
\end{proposition}
Both propositions are proven in Section~\ref{sec:moments}.

\remark{
When applied to the Example Model, Proposition~\ref{prop:Pkappa} states that $\E[(P^{(r)})^\kappa]$ is uniformly bounded from above by $C_2(1-r)^{1-2\kappa}$, for some constant $C_2$. Alternatively, the integer moments of the busy period duration in an $\mg$ queue can be calculated explicitly from its Laplace-Stieltjes transform, yielding $\E[P^{(r)}] = \frac{\E[B]}{1-r}$ and $\E[(P^{(r)})^2] = \frac{\E[B^2]}{(1-r)^3}$. One may therefore conclude that the asymptotic behaviour of the bound in Proposition~\ref{prop:Pkappa} is in fact sharp for the first two moments of the busy period duration $P^{(r)}$ in the Example Model.
}

From Propositions~\ref{prop:Pkappa} and \ref{prop:Nkappa} it follows that, for some constant $C_3$,
\begin{align*}
 \frac{1}{\E[N^{(n)}]} \E\left[ P^{(n)} N^{(n)} \ind(N^{(n)}>N^{(n)}_0)\right] \hspace{-40pt} & \\
  &\leq C_3 (1-\r^{(n)})^{\left(1-2\frac{s}{s-1}\right)\frac{s-1}{s}} (1-\r^{(n)})^{-\frac{3}{2}} (1-\r^{(n)})^{1-\frac{2-s}{2s}} (1-\r^{(n)})^{\frac{2-s}{2s}\zeta} \\
  &= C_3 (1-\r^{(n)})^{-1-\frac{1}{s}} (1-\r^{(n)})^{-\frac{3}{2}} (1-\r^{(n)})^{\frac{3}{2}-\frac{1}{s}} (1-\r^{(n)})^{\frac{2-s}{2s}\zeta} \\
  &= C_3 (1-\r^{(n)})^{\frac{2-s}{2s}\zeta-\frac{2+s}{s}}. \numberthis
\end{align*}
Comparing this to $\log(N^{(n)}_0) \E[T^{(n)}_{\srpt}]$ and noting that $\E[T^{(n)}_{\srpt}]\geq \E[B^{(n)}]$ yields
\begin{align*}
 \limsup_{\ntoinfty} \frac{\frac{1}{\E[N^{(n)}]} \E\left[ P^{(n)} N^{(n)} \ind(N^{(n)}>N^{(n)}_0)\right]}{\log(N_0^{(n)}) \E[T^{(n)}_{\srpt}]}
  &\leq \limsup_{\ntoinfty} \frac{C_3 (1-\r^{(n)})^{\frac{2-s}{2s}\zeta-\frac{2+s}{s}}}{\zeta \log\left(\frac{1}{1-\r^{(n)}}\right) \E[B^{(n)}]} \\
  &\leq \limsup_{\ntoinfty} \frac{C_3 (1-\r^{(n)})^{\frac{2-s}{2s}\zeta-\frac{2+s}{s}}}{\zeta \log\left(\frac{1}{1-\r^{(n)}}\right) \rho_0 A_{\min}},
\end{align*}
which tends to zero as $\ntoinfty$. This completes the proof of Theorem~\ref{thm:main}.

\section{Moments of busy period functionals}
\label{sec:moments}
This section proves several results on the moments of functionals. First, we introduce some new notation in Section~\ref{subsec:netputprocesses}. Then, we state and prove two lemmas in Section~\ref{subsec:preliminarylemmas} in order to prove Propositions~\ref{prop:Pkappa} and \ref{prop:Nkappa}. Sequentially, the propositions are proven in Sections~\ref{subsec:Pkappa} and \ref{subsec:Nkappa}. We emphasise that all of the functionals considered are independent of the scheduling policy, provided that it is work-conserving.

\subsection{Counting and netput processes}
\label{subsec:netputprocesses}
For any non-negative random variable $Y$, we define a random variable $Y_e$ that is distributed as the excess of $Y$; i.e.\ $\P(Y_e\leq x) = \int_0^x \P(Y>x) \dd x /\E[Y]$.
%
Next, we define two counting processes in the $\gg$ queue under consideration. The first process $N^{(n)}(t):=\inf\{m\in\{1,2,\ldots\}: A^{(n)}_1+\ldots +A^{(n)}_m \geq t\}, t\geq 0,$ counts the number of arrivals in $t$ time units, starting from a reference arrival that is also the first count. The second process $\tilde{N}^{(n)}(t), t\geq 0,$ is similar and only differs by initializing the count at an arbitrary point in time. Specifically,
\begin{equation*}
 \tilde{N}^{(n)}(t) = \left\{ \begin{array}{ll}
   0 & \text{if } t<A^{(n)}_e, \\
   \inf\{m\in\{1,2,\ldots\}: A^{(n)}_{e}+A^{(n)}_1+\ldots +A^{(n)}_m \geq t\} \qquad\, & \text{otherwise}.
  \end{array} \right.
\end{equation*}
These counting processes allow us to introduce two netput processes, $X^{(n)}(t) = \sum_{i=1}^{N^{(n)}(t)} B^{(n)}_i - t$ and $\tilde{X}^{(n)}(t) = \sum_{i=1}^{\tilde{N}^{(n)}(t)} B^{(n)}_i - t$, that quantify the net amount of work that could have been processed by the server in the $t$ time units after an arrival, or respectively after an arbitrary point in time
. Note that $X(t)$ becomes negative right after the first time that the queue is emptied. One may verify that $\P(\tilde{X}^{(n)}(t) > x) \leq \P(X^{(n)}(t) > x)$ for all $t\geq 0$, which will be denoted by $\tilde{X}^{(n)}(t)\leq_{st} X^{(n)}(t)$ in the remainder of this paper.

Similarly, we define two discrete processes that quantify the netput process \textit{at an arrival instance}, denoted by $S^{(n)}_m$ and $\tilde{S}^{(n)}_m, m\geq 0$. The process $S^{(n)}_m$ is defined as $S^{(n)}_0:=0, S^{(n)}_m:=\sum_{i=1}^m [B^{(n)}_i-A^{(n)}_i]$ and quantifies all the work that the server has received between the arrival of the reference job and the $m$-th next arrival, minus the work that it could have addressed during this time. The process $\tilde{S}^{(n)}_m$ starts observing at an arbitrary point in time instead of at the arrival of a reference job. It is defined as $\tilde{S}^{(n)}_0 = -A^{(n)}_e, \tilde{S}^{(n)}_m = -A^{(n)}_e+\sum_{i=1}^m [B^{(n)}_i-A^{(n)}_i]$. Again, we obtain the relation $\tilde{S}^{(n)}_m \leq_{st} S^{(n)}_m$. Also, one may verify that $\sup_{t\geq 0} X^{(n)}(t) = \sup_{m\in\{1,2,\ldots\}} S^{(n)}_m$, and hence by \citet[Cor.~III.6.5]{asmussen2003applied} we have $\sup_{t\geq 0} X^{(n)}(t) = \sup_{m\in\{0,1,\ldots\}} S^{(n)}_m \stackrel{d}{=} W^{(n)}$. 
Here, $\cdot \stackrel{d}{=} \cdot$ denotes equality in distribution. 
All sums $\sum_{i=1}^0$ are understood to be zero.


\subsection{Preliminary lemmas}
\label{subsec:preliminarylemmas}
The following two lemmas facilitate the proof of Propositions~\ref{prop:Pkappa} and \ref{prop:Nkappa}. Lemma~\ref{lem:IN} concerns the first moment of $N^{(n)}$ and $I^{(n)}$, whereas Lemma~\ref{lem:Mkappa} considers general moments of $W^{(n)}$.


\begin{lemma}
\label{lem:IN}
The relations
\begin{equation}
 (1-\r^{(n)}) \E[N^{(n)}] = \Theta(1) \text{ and } \E[I^{(n)}] = \Theta(1)
 \label{eq:IN}
\end{equation}
both hold as $n\rightarrow \infty$.
\end{lemma}

\begin{proof}[Proof of Lemma~\ref{lem:IN}.]
Since we have $\mu^{(n)} = \E[A^{(n)}](1-\r^{(n)}) = \Theta(1-\r^{(n)})$, it suffices to prove the relation $\mu^{(n)}\E[N^{(n)}] = \Theta(1)$. Proposition~X.3.1 in \citet{asmussen2003applied}, stating
\begin{equation}
 \E[I^{(n)}] = 
 \mu^{(n)}\E[N^{(n)}],
 \label{eq:relationIN}
\end{equation}
then implies that this is equivalent to the relation $\E[I^{(n)}]=\Theta(1)$.

Both the upper and the lower bound follow from \citet{lotov2002inequalities}, who considers the ladder height of a random walk. Specifically, 
Lotov obtains upper bounds for the moments of the ladder epochs and the moments of overshoot over an arbitrary non-negative level if the expectation of jumps is positive and close to zero. As such, his results apply to the random walk $-S(n)$ with ladder epochs $N^{(n)}$.

The upper bound is implied by Theorem~2 in \citet{lotov2002inequalities}, which claims that
\begin{equation}
 \mu^{(n)} \E[N^{(n)}] \leq C_4
\end{equation}
for some constant $C_4$ and all $n$, provided that $\sup_{n\in\{1,2,\ldots\}} \E[(\max\{A^{(n)}-B^{(n)},0\})^2]<\infty$. Accordance with this condition follows directly from $\sup_{n\in\{1,2,\ldots\}} \E[(A^{(n)})^2]<\infty$.

The lower bound is implied by inequality (2) in \citet{lotov2002inequalities}. In our model, we assumed that there exist constants $\d>0, \g>0$ such that $\P(B^{(n)} - A^{(n)} \geq \d)\geq \g$ for all $n$. \citet{lotov2002inequalities} then states
\begin{align*}
 \mu^{(n)}\E[N^{(n)}] \geq \int_0^\infty x \dd \P(B^{(n)}-A^{(n)} \leq x) \geq \d \g
\end{align*}
for all $n$. This completes the proof.
\end{proof}

\begin{lemma}
\label{lem:Mkappa}
 Let $p>0$ and define $q=\max\{2,p+1\}$. Assume that $\sup_{n\in\{1,2,\ldots\}} \E[(B^{(n)})^q] < \infty$. Then
\begin{equation}
 \limsup_{\ntoinfty} (1-\r^{(n)})^p\E[(W^{(n)})^p]<\infty.
 \label{eq:Mkappa}
\end{equation}
\end{lemma}

\remark{
 Consider the Example Model, and assume that jobs are served according to the $\fifo$ discipline. Then $W^{(r)}$ is just the waiting time of a job, and hence for some constant $C_5$ the average sojourn time $\E[T^{(r)}_\fifo] = \E[W^{(r)}_\fifo] + \E[B] \leq \frac{C_5}{1-r} + \E[B]$ scales no worse than $1/(1-r)$. 
 Lemma~\ref{lem:Mkappa} provides bounds on the work $W^{(r)}$ at an arrival for more general moments, provided that a sufficiently high moment of the job size distribution exists.
}

\begin{proof}[Proof of Lemma~\ref{lem:Mkappa}.]
Since $\sup_{n\in\{1,2,\ldots\}} \E[A^{(n)}]<\infty$, relation \eqref{eq:Mkappa} is equivalent to
\begin{equation}
 (\mu^{(n)})^p\E[(W^{(n)})^p]<\infty,
\end{equation}
which is proven below.

Assume $p\geq 1$ and let $E^{(n)}_i,i\in\{1,2,\ldots\},$ be independent exponentially distributed random variables with mean
\[\E[E^{(n)}] = \frac{\E[A^{(n)}]+\E[B^{(n)}]}{2} < \E[A^{(n)}].\] 
We define $E^{(n)} := E^{(n)}_1$ and note that $\sup_{n\in\{1,2,\ldots\}} \E[E^{(n)}] \leq \sup_{n\in\{1,2,\ldots\}} \E[A^{(n)}] < \infty$.

By \citet[Cor.~III.6.5]{asmussen2003applied} and subadditivity of suprema, $W^{(n)}$ is upper bounded as
 \begin{align*}
  W^{(n)} &\stackrel{d}{=} \sup_{m\in\{0,1,2,\ldots\}} \sum_{i=1}^m \left[ B^{(n)}_i - A^{(n)}_i \right] 
   \leq \sup_{m\in\{0,1,2,\ldots\}} \sum_{i=1}^m \left[ B^{(n)}_i - E^{(n)}_i \right] + \sup_{m\in\{0,1,2,\ldots\}} \sum_{i=1}^m \left[ E^{(n)}_i - A^{(n)}_i \right] \\
   &=: W^{(n)}_1 + W^{(n)}_2,
 \end{align*}
where $W^{(n)}_1$ can be interpreted as the total work in an $\mg$ queue as observed by an arrival, and $W^{(n)}_2$ as the total work in an $\gm$ queue as observed by an arrival. As a consequence,
$
 \P(W^{(n)} > x) \leq \P(W^{(n)}_1 +  W^{(n)}_2 > x) \leq \P(W^{(n)}_1 > x/2) + \P(W^{(n)}_2 > x/2)
$ and thus
\begin{align*}
 \E[(W^{(n)})^p] &= p \int_0^\infty x^{p-1} \P(W^{(n)} > x) \dd x \\
  &\leq p \int_0^\infty x^{p-1} \P(W^{(n)}_1 > x/2) \dd x + p \int_0^\infty x^{p-1} \P(W^{(n)}_2 > x/2) \dd x \\
  &= 2^p \left(\E[(W^{(n)}_1)^p]+\E[(W^{(n)}_2)^p]\right).
\end{align*}

First, we consider $W^{(n)}_1$. Define the geometrically distributed random variable $K^{(n)}_1$ with support $\{0,1,\ldots\}$ and fail parameter
\[\xi^{(n)}_1:=\frac{\E[B^{(n)}]}{\E[E^{(n)}]} = \frac{2\E[B^{(n)}]}{\E[A^{(n)}]+\E[B^{(n)}]}.\]
For notational convenience, we drop the superscript $(n)$ of $\xi_1^{(n)}$ for the remainder of this section. 
Theorem~VIII.5.7 in \citet{asmussen2003applied} presents a random sum representation of the functional $W^{(n)}_1$ in terms of $K^{(n)}_1$ and $B^{(n)}_{e,i}$:
\begin{equation*}
 W^{(n)}_1 \stackrel{d}{=} \sum_{i=1}^{K^{(n)}_1} B^{(n)}_{e,i}.
\end{equation*}
Since $f(x)=x^p$ is a convex function for all $p\geq 1$, Lemma~5 in \citet{remerova2014random} implies
\begin{equation}
 \E[(W^{(n)}_1)^p] \leq \E[(K^{(n)}_1)^p] \E[(B^{(n)}_{e})^p].
 \label{eq:W1p}
\end{equation}
The conditions of Lemma~\ref{lem:Mkappa} ensure that the $p$-th moment of $B^{(n)}_{e}$ is finite as $\ntoinfty$:
\begin{align*}
 \E[(B^{(n)}_{e})^p] &= \int_0^\infty x^p \dd \P(B^{(n)}_{e}\leq x) 
  = \frac{1}{\E[B^{(n)}]} \int_0^\infty x^p \P(B^{(n)}> x) \dd x \\
  &= \frac{1}{(p+1)\E[B^{(n)}]} \int_0^\infty x^{p+1} \dd \P(B^{(n)}\leq x)
  = \frac{1}{(p+1)\E[B^{(n)}]} \E[(B^{(n)})^{p+1}]. \numberthis \label{eq:Bep}
\end{align*}
Therefore, we need to show that $(\mu^{(n)})^p\E[(K^{(n)}_1)^p]$ is uniformly bounded as $\ntoinfty$.
Let $k=\lfloor p\rfloor$. Then
\begin{align*}
 \E[(K^{(n)}_1)^p] &= \frac{1-\xi_1}{(1-\xi_1)^p} \sum_{m=0}^\infty ((1-\xi_1)m)^p \xi_1^m \\
  &\leq \frac{1-\xi_1}{(1-\xi_1)^p} \sum_{m=0}^{\left\lfloor \frac{1}{1-\xi_1}\right\rfloor} ((1-\xi_1)m)^k \xi_1^m + \frac{1-\xi_1}{(1-\xi_1)^p} \sum_{m=\left\lfloor \frac{1}{1-\xi_1}\right\rfloor + 1}^\infty ((1-\xi_1)m)^{k+1} \xi_1^m \\
  &\leq \frac{(1-\xi_1)^{k+1}}{(1-\xi_1)^p} \sum_{m=0}^\infty m^k \xi_1^m
   + \frac{(1-\xi_1)^{k+2}}{(1-\xi_1)^p} \sum_{m=0}^\infty m^{k+1} \xi_1^m. \numberthis \label{eq:expectation}
\end{align*}
On the one hand, for any $\ell\in\{1,2,\ldots\}$, we have
\begin{align*}
 (1-\xi_1)^{\ell+1} \xi_1^\ell \frac{\dd^\ell}{\dd \xi_1^\ell} \sum_{m=0}^\infty \xi_1^m &= (1-\xi_1)^{\ell+1} \xi_1^\ell \frac{\dd^\ell}{\dd \xi_1^\ell} (1-\xi_1)^{-1} = \ell! \xi_1^\ell. \numberthis \label{eq:diffG1}
\end{align*}
On the other hand we have
\begin{align*}
 (1-\xi_1)^{\ell+1} \xi_1^\ell \frac{\dd^\ell}{\dd \xi_1^\ell} \sum_{m=0}^\infty \xi_1^m 
  &= (1-\xi_1)^{\ell+1} \sum_{m=\ell}^\infty m(m-1)\cdots(m-\ell+1) \xi_1^m \\
  &= (1-\xi_1)^{\ell+1} \sum_{m=0}^\infty m^\ell \xi_1^m - (1-\xi_1)^{\ell+1} \sum_{m=0}^{\ell-1} m^\ell \xi_1^m 
    + (1-\xi_1)^{\ell+1} \sum_{m=\ell}^\infty o(m^\ell) \xi_1^m. \numberthis \label{eq:diffG2}
\end{align*}
Combining equalities \eqref{eq:diffG1} and \eqref{eq:diffG2}, we find that
\[(1-\xi_1)^{\ell+1} \sum_{m=0}^\infty m^\ell \xi_1^m = \ell! \xi_1^\ell + (1-\xi_1)^{\ell+1} \sum_{m=0}^{\ell-1} m^\ell \xi_1^m
   + (1-\xi_1)^{\ell+1} \sum_{m=\ell}^\infty o(m^\ell) \xi_1^m. \]
Now, for any $\nu>0$ there exists a $M_\nu\in\{1,2,\ldots\}$ independent of the system index $n$ such that for all $m\geq M_\nu$ the $o(m^\ell)$ term is dominated by $\nu m^\ell$. Fix such $\nu\in(0,1)$ and $M_\nu$. Then
\begin{align*}
 (1-\xi_1)^{\ell+1} \sum_{m=0}^\infty m^\ell \xi_1^m &\leq \ell! + \ell^{\ell+1} + (1-\xi_1)^{\ell+1} \nu \sum_{m=M_\nu}^\infty m^\ell \xi_1^m + (1-\xi_1)^{\ell+1} \sum_{m=0}^{M_\nu} o(m^\ell) \xi_1^m \\
  &\leq C_6 + (1-\xi_1)^{\ell+1} \nu \sum_{m=0}^\infty m^\ell \xi_1^m
\end{align*}
for some constant $C_6>0$, and hence
\begin{equation}
 (1-\xi_1)^{\ell+1} \sum_{m=0}^\infty m^\ell \xi_1^m \leq \frac{C_6}{1-\nu}.
 \label{eq:geoSum}
\end{equation}
Since $(\mu^{(n)}/(1-\xi_1))^p = (\E[A^{(n)}]+\E[B^{(n)}])^p$, we may conclude from relations~\eqref{eq:expectation} and \eqref{eq:geoSum} that $(\mu^{(n)})^p \E[(K^{(n)}_1)^p]$ is uniformly bounded from above as $\ntoinfty$, and so is $(\mu^{(n)})^p \E[(W^{(n)}_1)^p]$ by \eqref{eq:W1p}.

Second, we consider the functional $W^{(n)}_2$. Recall that $W^{(n)}_2$ denotes the steady-state workload in a $\gm$ queue upon arrival. Theorem~VIII.5.8 and page 296 in \citet{asmussen2003applied} together state that
\begin{equation}
 W^{(n)}_2 \stackrel{d}{=} \sum_{i=1}^{K^{(n)}_2} E^{(n)}_i,
 \label{eq:W2}
\end{equation}
where $K^{(n)}_2$ is a geometrically distributed random variable with support $\{0,1,\ldots\}$ and unknown fail parameter $\xi^{(n)}_2$. \citet{remerova2014random} again ensure that
\begin{equation}
 \E[(W^{(n)}_2)^p] \leq \E[(K^{(n)}_2)^p] \E[(E^{(n)})^p],
 \label{eq:W2p}
\end{equation}
where the latter expectation is finite uniformly in $n$ as a property of exponential distributions. The $p$-th moment of $K^{(n)}_2$ is bounded by \eqref{eq:expectation} and \eqref{eq:geoSum}, so that
\begin{equation}
 (\mu^{(n)})^p \E[(K^{(n)}_2)^p] \leq C_7 \left(\frac{\mu^{(n)}}{1-\xi^{(n)}_2}\right)^p
\end{equation}
for some constant $C_7>0$. The proof is complete once we show $\frac{\mu^{(n)}}{1-\xi^{(n)}_2} = \O(1)$.

From \eqref{eq:W2} it can be verified that $\P(W^{(n)}_2=0) = \P(K^{(n)}_2=0) = 1-\xi^{(n)}_2$. Additionally, by Theorem~VIII.2.3 in \citet{asmussen2003applied}, we have $\P(W^{(n)}_2=0)=1/\E[N^{(n)}_2]$ and hence $\E[N^{(n)}_2] = 1/(1-\xi^{(n)}_2)$. Here, $N^{(n)}_2$ is the steady-state number of jobs in a busy period of the $\gm$ queue. Applying Lemma~\ref{lem:IN} with interarrival times $A^{(n)}_i$, job sizes $E^{(n)}_i$, and mean service between arrivals $\frac{1}{2}\mu^{(n)}$ yields $\frac{1}{2}\mu^{(n)}\E[N^{(n)}_2]=\Theta(1)$, and therefore $\mu^{(n)}/(1-\xi^{(n)}_2) = \O(1)$.

Finally, for $0<p<1$ the lemma follows directly from the case $p=1$ after observing that $(\mu^{(n)})^p\E[(W^{(n)})^p]\leq \left(\mu^{(n)}\E[W^{(n)}]\right)^p$ by Jensen's inequality.
\end{proof}

Lemmas~\ref{lem:IN} and \ref{lem:Mkappa} provide the asymptotic behaviour of functionals that are closely related to $P^{(n)}$ and $N^{(n)}$. The remainder of this section utilizes these results in order to prove Propositions~\ref{prop:Pkappa} and \ref{prop:Nkappa}.

\subsection[Busy period duration P]{Busy period duration $P^{(n)}$}
\label{subsec:Pkappa}
This section is devoted to the proof of Proposition~\ref{prop:Pkappa}. We wish to show that
\begin{equation*}
 \E[(P^{(n)})^\kappa] = \O\left((1-\r^{(n)})^{1-2\kappa} \right).
 \tag{\ref{eq:Pkappa}, revisited}
\end{equation*}
for all $\kappa \in [1,\alpha]$, provided that $\sup_{n\in\{1,2,\ldots\}} \E[(B^{(n)})^\alpha]<\infty$ for some $\a\geq 2$. Moreover, we claim that $\E[P^{(n)}] = \Theta \left((1-\r^{(n)})^{-1}\right)$.

First, consider $\kappa=1$. Due to Little's law for a busy server, we have
\begin{equation}
1-\r^{(n)} = \frac{\E[I^{(n)}]}{\E[I^{(n)}]+\E[P^{(n)}]},
\end{equation}
so that
\begin{equation}
\E[P^{(n)}] = \frac{\r^{(n)} \E[I^{(n)}]}{1-\r^{(n)}}.
\end{equation}
The result now follows from Lemma~\ref{lem:IN}.




Second, consider $\kappa>1$. Similar to \eqref{eq:Bep}, one obtains $\E[(P^{(n)}_e)^{\kappa-1}] = \E[(P^{(n)})^\kappa]/(\kappa \E[P^{(n)}])$ and hence it suffices to show that $\E[(P^{(n)}_e)^{\kappa-1}] = \O((1-\r^{(n)})^{2(1-\kappa)} )$. We have the following convenient representation for $P^{(n)}_e$ \citep[Thm.~X.3.4]{asmussen2003applied}: 
\begin{align*}
 P^{(n)}_e &\stackrel{d}{=} \inf\{ \tau \geq 0: \tilde{X}^{(n)}(\tau) \leq -V^{(n)} \mid V^{(n)}>0\} 
  \stackrel{d}{=} \inf\{ \tau \geq 0: B^{(n)}_e + W^{(n)} + \tilde{X}^{(n)}(\tau) \leq 0 \}.
\end{align*}
The above relation allows us to bound $\P(P^{(n)}_e>t)$ as
\begin{align*}
 \P(P^{(n)}_e>t) &= \P(\inf\{ \tau \geq 0: B^{(n)}_e + W^{(n)} + \tilde{X}^{(n)}(\tau) \leq 0 \}> t) \\
  &= \P(B^{(n)}_e + W^{(n)} + \tilde{X}^{(n)}(\tau) > 0 , \forall \tau\leq t) \\
  &\leq \P(B^{(n)}_e + W^{(n)} + \tilde{X}^{(n)}(t) + (1-\r^{(n)})t/2> (1-\r^{(n)})t/2) \\
  &\leq \P(B^{(n)}_e> (1-\r^{(n)})t/6) + \P(W^{(n)} > (1-\r^{(n)})t/6) \\
   &\qquad + \P\left(\sup_{\tau\geq 0} [\tilde{X}^{(n)}(\tau) + (1-\r^{(n)})\tau/2] > (1-\r^{(n)})t/6 \right).
\end{align*}
In Section~\ref{subsec:netputprocesses} we derived the relations $\tilde{X}^{(n)}(\tau) \leq_{st} X^{(n)}(\tau)$ and $W^{(n)} \stackrel{d}{=} \sup_{\tau\geq 0} X^{(n)}(\tau)$. These relations imply
\begin{equation}
 \P(P^{(n)}_e>t) \leq \P(B^{(n)}_e>(1-\r^{(n)})t/6) 
  + 2\P\left(\sup_{\tau\geq 0} [X^{(n)}(\tau)+ (1-\r^{(n)})\tau/2] > (1-\r^{(n)})t/6\right).
 \label{eq:Pkappaparts}
\end{equation}
The last inequality suggests that
\begin{align*}
 \E[(P^{(n)}_e)^{\kappa-1}] &= (\kappa-1)\int_0^\infty t^{\kappa-2} \P(P^{(n)}_e>t) \dd t \\
  &\leq (\kappa-1)\int_0^\infty t^{\kappa-2} \P(B^{(n)}_e>(1-\r^{(n)})t/6) \dd t \\
   &\qquad + 2(\kappa-1)\int_0^\infty t^{\kappa-2} \P\left(\sup_{\tau\geq 0} [X^{(n)}(\tau)+ (1-\r^{(n)})\tau/2] > (1-\r^{(n)})t/6\right) \dd t.
\end{align*}

To deal with the first term, note that
\begin{align*}
 (\kappa-1)\int_0^\infty t^{\kappa-2}\P(B^{(n)}_e>(1-\r^{(n)})t/6) \dd t \hspace{-150pt} & \hspace{150pt}
  = (\kappa-1)\int_0^\infty t^{\kappa-2}\P(6B^{(n)}_e/(1-\r^{(n)}) >t) \dd t \\
  &= \E\left[\left(\frac{6B^{(n)}_e}{1-\r^{(n)}}\right)^{\kappa-1}\right]
  = \O\left((1-\r^{(n)})^{1-\kappa}\right) \E[(B^{(n)}_e)^{\kappa-1}]
  = o\left((1-\r^{(n)})^{2(1-\kappa)}\right),
\end{align*}
since $\E[(B^{(n)})^{\kappa}]<\infty$ implies $\E[(B^{(n)}_e)^{\kappa-1}]<\infty$ (cf.\ \ref{eq:Bep}).

For the second term, observe that
\begin{align*}
\sup_{\tau\geq 0} [X^{(n)}(\tau)+ (1-\r^{(n)})\tau/2] \hspace{-44pt} & \hspace{44pt}
  = \sup_{\tau\geq 0} \left[\sum_{i=1}^{N^{(n)}(\tau)}B^{(n)}_i-\tau+ (1-\r^{(n)})\tau/2 \right] \\
  &= \sup_{\tau\geq 0} \left[\sum_{i=1}^{N^{(n)}(\tau)}B^{(n)}_i-\frac{1+\r^{(n)}}{2}\tau \right] 
  = \sup_{n\in \{0,1,2,\ldots\}} \left[\sum_{i=1}^{n} \left\{B^{(n)}_i-\frac{1+\r^{(n)}}{2} A^{(n)}_i\right\} \right] \\
  &=: \tilde{W}^{(n)},
\end{align*}
where $\tilde{W}^{(n)}$ is equal in distribution to the steady-state cumulative amount of work at arrival in a $\gg$ queue with job sizes $B^{(n)}_i$ and interarrival times $\frac{1+\r^{(n)}}{2}A^{(n)}_i,i\in\{1,2,\ldots\}$. The mean amount of work that the server completes between two consecutive arrivals in this system is then given by $\tilde{\mu}^{(n)}:=\frac{1+\r^{(n)}}{2}\E[A^{(n)}] - \E[B^{(n)}] = \frac{1-\r^{(n)}}{2}\E[A^{(n)}]$.
We therefore obtain
\begin{align*}
 (\kappa-1) \int_0^\infty t^{\kappa-2} \P\left(\sup_{\tau\geq 0} [X^{(n)}(\tau)+ (1-\r^{(n)})\tau/2] > (1-\r^{(n)})t/6\right) \dd t \hspace{-92pt} & \\
  &= (\kappa-1)\int_0^\infty t^{\kappa-2} \P\left(\frac{6}{1-\r^{(n)}} \tilde{W}^{(n)} > t\right) \dd t \\
  &= \frac{6^{\kappa-1}}{(1-\r^{(n)})^{\kappa-1}(\tilde{\mu}^{(n)})^{\kappa-1}} \E[(\tilde{\mu}^{(n)}\tilde{W}^{(n)})^{\kappa-1}] \\
  &\leq \frac{C_8}{(1-\r^{(n)})^{2(\kappa-1)}} \E[(\tilde{\mu}^{(n)}\tilde{W}^{(n)})^{\kappa-1}]
\end{align*}
for some constant $C_8>0$, by using Lemma~\ref{lem:IN}. Finally, $\E[(\tilde{\mu}^{(n)}\tilde{W}^{(n)})^{\kappa-1}]$ is bounded due to Lemma~\ref{lem:Mkappa}, which completes the proof of the proposition.

\subsection[Arrivals in a busy period N]{Arrivals in a busy period $N^{(n)}$}
\label{subsec:Nkappa}
This section contains the proof of Proposition~\ref{prop:Nkappa}. The proposition states that
\begin{equation*}
 \E[(N^{(n)})^\kappa] = \O\left((1-\r^{(n)})^{1-2\kappa}\right).
 \tag{\ref{eq:Nkappa}, revisited}
\end{equation*}
for all $\kappa \in [1,\alpha]$, provided that $\sup_{n\in\{1,2,\ldots\}} \E[(B^{(n)})^\alpha]<\infty$ for some $\a\geq 2$. Moreover, we claim that $\E[N^{(n)}] = \Theta \left((1-\r^{(n)})^{-1}\right)$.

The structure of the proof is identical to the proof of Proposition~\ref{prop:Pkappa}. 
For $\kappa=1$, the result follows directly from Lemma~\ref{lem:IN}. Therefore, we consider $\E[(N^{(n)})^\kappa]$ for $\kappa > 1$.
Similar to the proof of Proposition~\ref{prop:Pkappa}, we use the relation
\begin{equation}
 \E[(N^{(n)}_e)^{\kappa-1}]=\frac{\E[(N^{(n)})^\kappa]}{\kappa \E[N^{(n)}]}
\end{equation}
and note that
\begin{align*}
 N^{(n)}_e &\stackrel{d}{=} \inf \{\eta\in \{0,1,2,\ldots\}: \tilde{S}^{(n)}_\eta \leq -V^{(n)} \mid V^{(n)}>0\} 
  \stackrel{d}{=} \inf \{\eta\in \{0,1,2,\ldots\}: B^{(n)}_e+ W^{(n)}+\tilde{S}^{(n)}_\eta \leq 0\}.
\end{align*}
As before, the relations $\tilde{S}^{(n)}_\eta \leq_{st} S^{(n)}_\eta$ and $W^{(n)} \stackrel{d}{=} \sup_{\eta\in\{0,1,\ldots\}} S^{(n)}_\eta$ are exploited in order to obtain an equivalent of~\eqref{eq:Pkappaparts}:
\begin{equation*}
 \P(N^{(n)}_e>m) \leq \P(B^{(n)}_e> (1-\r^{(n)})m/6) + 2\P\left(\sup_{\eta\in \{0,1,2,\ldots\}} [S^{(n)}_\eta + (1-\r^{(n)})\eta/2] > (1-\r^{(n)})m/6 \right).
\end{equation*}
For any $\eta\in\{0,1,2,\ldots\}$, consider $\tau^{(n)}_\eta := A^{(n)}_1 + \ldots + A^{(n)}_\eta$. Then $S^{(n)}_\eta = X^{(n)}(\tau^{(n)}_\eta)$, so in particular
\begin{equation*}
 \P(N^{(n)}_e>m) \leq \P(B^{(n)}_e>(1-\r^{(n)})m/6) + 2\P\left(\sup_{\tau \geq 0} [X^{(n)}(\tau) + (1-\r^{(n)})\tau/2] > (1-\r^{(n)})m/6\right).
\end{equation*}
The remainder of the proof is identical to that of Proposition~\ref{prop:Pkappa}.

\section{Conclusion}
\label{sec:conclusion}
In this paper, we proved a result about the average case performance of (an extension of) the Randomized Multilevel Feedback ($\rmlf$) algorithm in a $\gg$ queue. Specifically, the gap in average sojourn time between the $\rmlf$ algorithm and the Shortest Remaining Processing Time algorithm behaves like $\O(\log (1/(1-\r^{(n)})))$ and this bound is tight for the $\mm$ queue. An appealing property of the $\rmlf$ algorithm is that its implementation does not depend on the input distributions $F^{(n)}_A$ and $F^{(n)}_B$; however, if $F^{(n)}_A$ and $F^{(n)}_B$ are known then there can be blind algorithms with a better performance than $\rmlf$ (e.g.\ Foreground-Background if $F^{(n)}_B$ has decreasing failure rate). 
The result was established by using techniques 
from both competitive analysis and applied probability. As the structure of the proof is quite general, it would be interesting to explore other possibilities in the intersection of these areas.

\section*{Acknowledgements}
This work is part of the free competition research programme with project number 613-001-219, which is financed by the Netherlands Organisation for Scientific Research (NWO). We would like to thank Adam Wierman for his encouragements to work on this research topic.

\appendix
The $\dirmlf$ algorithm is presented after the introduction of some notation. Define the queues $\tilde{Q}_z, z\in\Z$ and a ``new job'' queue $\tilde{Q}^*$. Let the targets $\tilde{U}_{z,j}$ be given by $\tilde{U}_{z,j} = 2^{z}\max\{1,2-\tilde{\beta}_j\}$, where the $\tilde{\beta}_j$'s are independent random variables with exponential cumulative distribution function $\P(\tilde{\beta}_j \leq x) = 1-\exp[-\theta x \ln j]$. Identical to the $\rmlf$ algorithm, $\theta$ is a symbolic constant fixed at $\theta=12$. All symbols $\tilde{J}_j, \tilde{r}_j, \tilde{B}_j$ and $\tilde{w}_j(t)$ are defined analogue to the symbols without accent in the $\rmlf$ algorithm. All release times $\tilde{r}_j$ must be distinct (e.g.\ all interarrival times are strictly positive), and jobs may have any size $\tilde{B}_j\geq 0$. Note that the original $\rmlf$ algorithm requires the job sizes to be uniformly bounded from below, but does not restrict the interarrival times to be non-zero.

Every job $\tilde{J}_h$ is assigned an initial target $\tilde{U}_{*,h}$ upon arrival, after which it is immediately served in $\tilde{Q}^*$ by a dedicated server. It departs from $\tilde{Q}^*$ on three occasions:
\begin{itemize}
\item The amount of service received equals the size $\tilde{B}_h$ of the job. In this case, $\tilde{J}_h$ is completed and leaves the system.
\item A new job enters the system. In this case, $\tilde{J}_h$ is moved to a queue $\tilde{Q}_{z},z\in\Z,$ that it naturally belongs to based on the amount of service $\tilde{w}_h(t)$ it has obtained thus far; that is, it is moved to the unique queue $\tilde{Q}_{z^*_h}$ that satisfies $\tilde{U}_{z^*_h-1,h} \leq \tilde{w}_h(t) < \tilde{U}_{z^*_h,h}$.
\item The amount of service received equals the initial target $\tilde{U}_{*,h}$. In this case, similar to the previous case, $\tilde{J}_h$ is moved to a queue $\tilde{Q}_{z}$ that it naturally belongs to.
\end{itemize}
The choice of the initial target $\tilde{U}_{*,h}$ depends on the system state:
\begin{itemize}
\item If the system is empty upon arrival, then the server is dedicated to $\tilde{J}_h$ regardless of the queue that $\tilde{J}_h$ is in. In this case, the target can be chosen arbitrarily; we set it to $\tilde{U}_{*,h}=\tilde{U}_{0,h}$.
\item If the system is not empty upon arrival, then there must be a lowest-index non-empty queue $\tilde{Q}_{z^*_h}$ (possibly after moving the job originally in $\tilde{Q}^*$ to another queue). $\tilde{J}_h$ may now experience a dedicated server until the moment when it would enter queue $\tilde{Q}_{z^*_h}$ based on its obtained service and the $(z^*_h-1)$-th target $\tilde{U}_{z^*_h-1,h}$. Therefore, $\tilde{J}_h$ should be moved no later then after $\tilde{U}_{*,h}=\tilde{U}_{z^*_h-1,h}$ units of obtained service.
\end{itemize}
If $\tilde{Q}^*$ is empty, then $\dirmlf$ always works on the non-empty queue $\tilde{Q}_z$ with the lowest index $z\in \Z$. Jobs in $\tilde{Q}_z$ are always served in a First Come First Serve order.

The $\dirmlf$ is formally presented in Figure~\ref{fig:dirmlf}.
\begin{figure}[!ht]
 \begin{center}
 \fbox{
\begin{minipage}{.85\textwidth}
\small{\textbf{Algorithm $\dirmlf$:} At all times the collection of released, but uncompleted, jobs are partitioned into queues, $\tilde{Q}^*,\tilde{Q}_z, z\in \Z$. We say that $\tilde{Q}_i$ is lower than $\tilde{Q}_j$ for $i<j$. $\tilde{Q}^*$ is the lowest queue. For each job $\tilde{J}_j\in \tilde{Q}_i, \tilde{U}_{i,j}\in [2^i,2^{i+1}]$ when it entered $\tilde{Q}_i$. $\dirmlf$ maintains the invariant that it is always running the job at the front of the lowest non-empty queue.
\\
When a job $\tilde{J}_h$ is released at time $\tilde{r}_h$, $\dirmlf$ takes the following actions:
\begin{itemize}
\item If, just prior to $\tilde{r}_h$, all queues were empty, then
 \begin{itemize}
 \item Job $\tilde{J}_{h}$ is enqueued on $\tilde{Q}^*$.
 \item The initial target $\tilde{U}_{*,h}$ is set to $\tilde{U}_{0,h}=\max\{1,2-\tilde{\beta}_h\}$.
 \end{itemize}
\item If, just prior to $\tilde{r}_h$, there are unfinished jobs in the system but $\tilde{Q}^*$ is empty, then
 \begin{itemize}
 \item Job $\tilde{J}_{h}$ is enqueued on $\tilde{Q}^*$.
 \item The initial target $\tilde{U}_{*,h}$ is set to $\tilde{U}_{z^*_h-1,h}=2^{z^*_h-1}\max\{1,2-\tilde{\beta}_h\}$, where the queue index $z^*_h=\min\{z\in \Z: \tilde{Q}_z \text{ non-empty at time } t\}$ corresponds to the lowest non-empty queue.
 \end{itemize}
\item If, just prior to $\tilde{r}_h$, $\tilde{Q}^*$ is non-empty, then $\tilde{Q}^*=\{\tilde{J}_{h-1}\}$ at that time. Now,
 \begin{itemize}
 \item The target $\tilde{U}_{z^*_h,h-1}=2^{z^*_h}\max\{1,2-\tilde{\beta}_{h-1}\}$ with $z^*_h:= \min\{z\in \Z: \tilde{w}_{h-1}(\tilde{r}_h) \leq \tilde{U}_{z,h-1}\} \}$ is the lowest target not yet reached by job $\tilde{J}_{h-1}$.
 \item Job $\tilde{J}_{h-1}$ is dequeued from $\tilde{Q}^*$.
 \item Job $\tilde{J}_{h-1}$ is enqueued on $\tilde{Q}_{z^*_h}$.
 \item Job $\tilde{J}_{h}$ is enqueued on $\tilde{Q}^*$.
 \item The initial target $\tilde{U}_{*,h}$ is set to $\tilde{U}_{z^*_h-1,h}=2^{z^*_h-1} \max\{1,2-\tilde{\beta}_h\}$.
 \end{itemize}
\item If, just prior to $\tilde{r}_h$, it was the case that $\dirmlf$ was running a job $\tilde{J}_j$, then $\tilde{J}_j$ is pre-empted.
\item $\dirmlf$ begins running $\tilde{J}_h$.
\end{itemize}
If at some time $t$, a job $\tilde{J}_j\in \tilde{Q}_{z-1}$ is being run when $\tilde{w}_j(t)$ becomes equal to $\tilde{U}_{z-1,j}$, then $\dirmlf$ takes the following actions:
\begin{itemize}
\item Job $\tilde{J}_j$ is dequeued from $\tilde{Q}_{z-1}$.
\item Job $\tilde{J}_j$ is enqueued on $\tilde{Q}_z$.
\item The target $\tilde{U}_{z,j}$ is set to $2\tilde{U}_{z-1,j} = 2^z\max\{1,2-\tilde{\beta}_j\}$.
\end{itemize}
If at some time $t$, a job $\tilde{J}_j\in \tilde{Q}^*$ is being run when $\tilde{w}_j(t)$ becomes equal to $\tilde{U}_{*,j}$, then $\dirmlf$ takes the following actions:
\begin{itemize}
\item Job $\tilde{J}_j$ is dequeued from $\tilde{Q}^*$.
\item Job $\tilde{J}_j$ is enqueued on $\tilde{Q}_{z^*_h}$, where $z^*_h = \log_2 \left(\tilde{w}_j(t)/\max\{1,2-\tilde{\beta}_j\}\right)+1$.
\item The target $\tilde{U}_{z^*_h,j}$ is set to $2\tilde{U}_{*,j} = 2^{z^*_h}\max\{1,2-\tilde{\beta}_j\}$.
\end{itemize}
Whenever a job is completed, it is removed from its queue.}
\end{minipage}
 }
 \end{center}
\caption{Formal statement of $\dirmlf$ algorithm.}\label{fig:dirmlf}
\end{figure}
Observe that both $\rmlf$ and $\dirmlf$ preserve the ordering of the jobs; that is, if job $\tilde{J}_j$ is released prior to job $\tilde{J}_k$ then as long as both jobs are incomplete:
\begin{itemize}
 \item job $\tilde{J}_j$ will never be in a lower queue than job $\tilde{J}_k$, and
 \item if both jobs are in the same queue, then job $\tilde{J}_j$ is closer to the front of the queue than job $\tilde{J}_k$.
\end{itemize}



We are now ready to prove Theorem~\ref{thm:dirmlf}, stating that
 \begin{equation*}
   \E[T_\dirmlf(\mathcal{I})] \leq C_1 \log(m) \E[T_\srpt(\mathcal{I})]
   \tag{\ref{eq:dirmlf}, revisited}
 \end{equation*}
 for all instances $\mathcal{I}$ of size at most $m$ for a universal constant $C_1$. This constant is identical to the constant $C_1$ in Theorem~\ref{thm:rmlf}.

\begin{proof}[Proof of Theorem~\ref{thm:dirmlf}.]
Consider any instance $\tilde{\mathcal{I}}$ for $\dirmlf$ of size at most $m$. 
Since all jobs of size zero are immediately served in queue $\tilde{Q}^*$ upon arrival, we assume without loss of generality that the instance does not contain any jobs of size zero. As a consequence, the minimum job size $\tilde{B}_{\min} = \min_{j=1,\ldots,|\tilde{\mathcal{I}}|} \tilde{B}_j$ is strictly positive. 
We now transform the instance $\tilde{\mathcal{I}}$ for $\dirmlf$ to a corresponding instance $\mathcal{I}$ for $\rmlf$.

Define the scaling parameter $g := \lfloor \log_2 (\tilde{B}_{\min})\rfloor - 1 \in \Z$, satisfying $2^{-g} \tilde{B}_{\min} \geq 2$. The instance $\mathcal{I}$ consists out of $|\tilde{\mathcal{I}}|$ jobs that are scaled versions of the original $|\tilde{\mathcal{I}}|$ jobs; specifically, job $J_j$ has size $B_j := 2^{-g} \tilde{B}_j$ and release date $r_j := 2^{-g} \tilde{r}_j$. Then, the smallest job is of size at least $2$ and the $\rmlf$ algorithm may be applied to the instance $\mathcal{I}$.

Since the jobs are released in the same order as in the original instance, we note that the random variables $\beta_j$ assigned by $\rmlf$ have the same distribution as the $\tilde{\beta}_j$ assigned by $\dirmlf$. We therefore couple these random variables in a trivial way: $\beta_j \equiv \tilde{\beta}_j$ for all $j=1,\ldots,|\tilde{\mathcal{I}}|$. It immediately follows that the targets $\tilde{U}_{z,j}$ as assigned to $\tilde{\mathcal{I}}$ by $\dirmlf$ and the targets $U_{i,j}$ as assigned to $\mathcal{I}$ by $\rmlf$ satisfy $\tilde{U}_{i,j} = U_{i,j}$ for all $i\in\{0,1,2,\ldots\}$. Additionally, the initial $\rmlf$ target $U_{0,j}$ satisfies $U_{0,j}=\tilde{U}_{0,j} = 2^{g-g}\max\{1,2-\tilde{\beta}_j\} = 2^{-g} \tilde{U}_{g,j}$.

We will show that the above construction implies an equivalence between $\rmlf$ and $\dirmlf$. For all $z\in \Z$ and $t\geq 0$ define the sets
\begin{equation}
 \tilde{Q}_z(t) := \left\{\tilde{J}_j: \tilde{U}_{z-1,j} \leq \tilde{w}_j(t) < \tilde{U}_{z,j} \right\}
\end{equation}
that contain all jobs in the system at time $t$ that are in queue $\tilde{Q}_z$, \textit{or the most recently released job if it has received a similar amount of service}. The equivalence is first observed between the initial $\rmlf$ queue $Q_0(t)$ and the augmented $\dirmlf$ queue $\hat{Q}(t)$, defined as
\begin{equation}
 Q_0(t) := \left\{J_j: w_j(t) < U_{0,j} \right\}
\end{equation}
and
\begin{equation}
 \hat{Q}(t) := \bigcup_{z=-\infty}^g \tilde{Q}_{z}(t) = \left\{\tilde{J}_j: \tilde{w}_j(t) < \tilde{U}_{g,j} \right\},
\end{equation}
respectively. One may observe that since $Q_0$ is the highest priority queue, it experiences a dedicated, work-conserving server that works at unit speed on jobs $J_j$ with sizes $U_{0,j}$. Therefore, the event that the $\rmlf$ server works on $Q_0$ is equivalent to the event $\{Q_0(t)>0\}$.

We assume without loss of generality that job $1$ arrives at time $r_1=0$. The arrival of this job initiates a first busy period for $Q_0(t)$ of $N_1$ jobs, where $N_1$ is such that the cumulative targets $U_{0,j}$ of the first $N_1$ jobs can be served before the $(N_1+1)$-th job is released. It is defined as $N_1 = \inf\{k\geq 1: \sum_{j=1}^k U_{0,j} - r_{j+1} \leq 0\}$, where $r_{|\tilde{\mathcal{I}}|+1}$ is understood as plus infinity. The duration of the busy period is given by $P_1 = \sum_{j=1}^{N_1} U_{0,j} = \sum_{j=1}^{N_1} 2^{-g} \tilde{U}_{g,j}$. The server may then work on jobs in higher queues (perceived as idle time by $Q_0$) until time $r_{N_1+1}$, when a new busy period is initiated. For $t\in[0,r_{N_1+1})$ we have now obtained
\begin{equation}
 Q_0(t) > 0 \Leftrightarrow t \leq P_1 \Leftrightarrow 2^{g} t \leq \sum_{j=1}^{N_1} \tilde{U}_{g,j}.
\end{equation}
By a similar analysis of the augmented queue $\hat{Q}$ we find that for all $t\in[0,2^g r_{N_1+1})=[0,\tilde{r}_{N_1+1})$ the relation $\hat{Q}(t) > 0 \Leftrightarrow Q_0(2^{-g} t) > 0$ holds, and for all $t\geq 0$ by a straightforward generalisation of the above procedure. Observing that both algorithms preserve the ordering of jobs, we may similarly show that the $\dirmlf$ server processes job $\tilde{J}_j$ in queue $\tilde{Q}_{g+i}$ at time $t$ if and only if the $\rmlf$ server processes job $J_j$ in queue $Q_i$ at time $2^g t$ for all $i\in\{1,2,\ldots\}$ and $t\geq 0$.

From the above results, one may deduce that the average sojourn time $\E[T_\dirmlf(\tilde{\mathcal{I}})]$ of instance $\tilde{\mathcal{I}}$ under algorithm $\dirmlf$ equals $2^g$ times the average sojourn time $\E[T_\rmlf(\mathcal{I})]$ of instance $\mathcal{I}$ under $\rmlf$. The competitive ratio of $\rmlf$ as stated in Theorem~\ref{thm:rmlf} hence guarantees that, for all instances $\tilde{\mathcal{I}}$ of size at most $m$,
\begin{equation}
 \E[T_\dirmlf(\tilde{\mathcal{I}})] = 2^g \E[T_\rmlf(\mathcal{I})] \leq C_1 \log(m) 2^g \E[T_\srpt(\mathcal{I})].
\end{equation}
The competitive ratio of $\dirmlf$ is concluded by verifying
\begin{equation}
 2^g \E[T_\srpt(\mathcal{I})] = \E[T_\srpt(\tilde{\mathcal{I}})],
\end{equation}
which is a direct consequence of our scaling. 
In particular, the constant $C_1$ in the upper bound is the same for $\rmlf$ and $\dirmlf$.
\end{proof}

\DeclareRobustCommand{\NLprefix}[3]{#3}.
\bibliographystyle{apalike}

\end{document}